\documentclass[11pt]{amsart}

\usepackage[T1]{fontenc}
\usepackage{lmodern}
\usepackage{amsmath,amssymb,mathtools}
\usepackage{graphicx}
\usepackage{tikz}
\usepackage{microtype}
\usepackage{flafter}
\usepackage{placeins}
\usepackage{amsrefs}
\usepackage{hyperref}
\usepackage[nameinlink,capitalize]{cleveref}
\usepackage{xcolor}

\newtheorem{theorem}{Theorem}[section]
\newtheorem{proposition}[theorem]{Proposition}
\newtheorem{lemma}[theorem]{Lemma}
\newtheorem{corollary}[theorem]{Corollary}
\theoremstyle{remark}
\newtheorem{remark}[theorem]{Remark}

\DeclareMathOperator{\graph}{graph}
\DeclareMathOperator{\ext}{ext}
\newcommand{\R}{\mathbb R}
\newcommand{\HH}{\mathfrak O}
\newcommand{\BH}{\mathsf B_{\HH}}
\newcommand{\SH}{\mathsf S_{\HH}}
\newcommand{\normH}[1]{\lVert #1\rVert_{\HH}}

\title[Quadratic forms on an octagonal sector]
{The Unit Ball of Quadratic Forms on an Octagonal Sector}

\author[Han]{Manwook Han}
\address[Manwook Han]{\mbox{}\newline\indent  Department of Mathematics,\newline\indent Chungbuk National University\newline\indent Cheongju, Chungbuk 28644, \newline\indent Republic of Korea}
\email{mwhan0828@gmail.com}

\author[Kim]{Sun Kwang Kim}
\address[Sun Kwang Kim]{\mbox{}\newline\indent  Department of Mathematics,\newline\indent Chungbuk National University\newline\indent Cheongju, Chungbuk 28644, \newline\indent Republic of Korea}
\email{skk@chungbuk.ac.kr}

\author[Mu\~noz-Fern\'andez]{Gustavo A. Mu\~noz-Fern\'andez}
\address[Gustavo A. Mu\~noz-Fern\'andez]{\mbox{}\newline\indent  Instituto de Matem\'atica Interdisciplinar (IMI) \newline\indent and \newline\indent Departamento de An\'alisis Matem\'atico y Matem\'atica Aplicada, \newline\indent Facultad de Ciencias Matem\'aticas, \newline\indent Universidad Complutense de Madrid, \newline\indent Plaza de Ciencias 3, \newline\indent 28040 Madrid, Spain}
\email{gustavo\_fernandez@mat.ucm.es}

\author[Seoane]{Juan B. Seoane--Sep\'ulveda}
\address[Juan B. Seoane--Sep\'ulveda]{\mbox{}\newline\indent  Instituto de Matem\'atica Interdisciplinar (IMI) \newline\indent and \newline\indent Departamento de An\'alisis Matem\'atico y Matem\'atica Aplicada, \newline\indent Facultad de Ciencias Matem\'aticas, \newline\indent Universidad Complutense de Madrid, \newline\indent Plaza de Ciencias 3, \newline\indent 28040 Madrid, Spain}
\email{jseoane@ucm.es}


\begin{document}

\begin{abstract}
Let
\[
 \HH=\{(x,y)\in[0,1]^2:x+y\le \sqrt2\}
\]
be the first-quadrant sector of a regular octagon.  For quadratic forms
\(P(x,y)=ax^2+bxy+cy^2\), we study the supremum norm over \(\HH\).  We obtain a complete five-region formula for the norm, according to whether the norming contact occurs at an endpoint or in the interior of one of the three radial sides.  We then prove that the projection of the unit ball onto the \(ac\)-plane is exactly \([-1,1]^2\), compute both endpoints of every vertical section, and thereby parametrize the entire unit sphere.  Finally, we characterize the extreme points of the unit ball as four explicit curves, their negatives, and four pairs of isolated points.  The resulting description is fully explicit and reduces subsequent convex extremal problems on this polynomial space to four one-parameter families and finitely many isolated polynomials.
\end{abstract}

\maketitle

\tableofcontents

\section{Introduction and preliminaries}

The geometry of finite-dimensional polynomial spaces is one of the classical meeting points of convexity, approximation theory, and Banach space methods.  Already in one variable, the structure of polynomial unit balls and their extreme points leads to subtle extremal problems; see the pioneering paper of Konheim and Rivlin \cite{KoRi}.  In the quadratic case, explicit norm formulas go back to Aron and Klimek \cite{AK}, while the unit balls of spaces of homogeneous polynomials on classical two-dimensional Banach spaces were described by Choi and Kim and by Grecu \cites{CK,Grecu}.  These results illustrate a recurring theme: whenever the norm can be written in a genuinely explicit way, one gains access not only to the unit ball itself, but also to sharp polynomial inequalities and to the extremal structure behind them.

This program becomes markedly richer on nonsymmetric convex bodies.  In that setting, the polynomial norm is no longer governed by absolute symmetries, and the regions controlling the norm often have to be identified one by one.  Complete geometric descriptions are known for several remarkable families, including nonsymmetric planar convex bodies \cite{MRS}, the unit square and related polynomial inequalities \cite{GMSS}, circular sectors \cites{BMRS,AJMS}, spaces of homogeneous trinomials \cites{JMR,GMS}, and hexagonal norms on the plane \cites{KimHex,KimExp}.  Closely related polynomial and multilinear geometries for octagonal norms have also been investigated in \cites{KimOctPoly,KimOctBil,KimOctGeom,KimOctNorming}.  General background and further applications can be found in the recent monographs \cites{FGMMRS,GJMMMS}.

A second motivation comes from the Krein--Milman method.  Once the extreme points of a polynomial unit ball are known, every continuous convex functional attains its maximum at those points.  This reduction has proved highly effective in the study of sharp Bernstein--Markov, polarization, unconditional, and Bohnenblust--Hille type inequalities; see, for example, \cites{AMRS,GMSU,JMMS}.  Obtaining a workable description of the extreme points is therefore not a purely geometric exercise: it is the key step that makes subsequent extremal applications feasible.

In this paper we study the space of quadratic forms on the nonsymmetric convex body
\begin{equation}\label{eq:defH}
 \HH=\{(x,y)\in[0,1]^2:x+y\le\sqrt2\}.
\end{equation}
The centrally symmetric polygon
\[
 \widetilde{\mathfrak O}
 :=\{(x,y)\in\R^2:|x|\le1,\ |y|\le1,\ |x|+|y|\le\sqrt2\}
\]
is a regular octagon, and \(\HH=\widetilde{\mathfrak O}\cap[0,\infty)^2\).  Thus \eqref{eq:defH} is precisely its first-quadrant sector.  If
\[
 \tau:=\sqrt2-1,
\]
then \(\tau^{-1}=\sqrt2+1\), \(\tau^2=3-2\sqrt2\), and \(1-\tau^2=2\tau\).  The non-axial part of the boundary of \(\HH\) consists of the three segments
\begin{align*}
 L_1&=\{(x,1):0\le x\le\tau\},\\
 L_2&=\{(x,\sqrt2-x):\tau\le x\le1\},\\
 L_3&=\{(1,y):0\le y\le\tau\}.
\end{align*}
They are the only places where a nonzero homogeneous quadratic form can attain its norm.  The geometry of the body is shown in \cref{fig:domain}.

For
\[
 P(x,y)=ax^2+bxy+cy^2,
\]
we consider the supremum norm
\[
 \normH{P}:=\sup\{|P(x,y)|:(x,y)\in\HH\}.
\]
Identifying \(P\) with \((a,b,c)\in\R^3\), we denote by
\[
 \BH=\{(a,b,c):\normH{(a,b,c)}\le1\},\qquad
 \SH=\{(a,b,c):\normH{(a,b,c)}=1\}
\]
the corresponding unit ball and unit sphere, and by
\(\pi_{ac}(a,b,c)=(a,c)\) the projection onto the \(ac\)-plane.

Our main results provide a complete geometric description of this norm.  First, we derive an explicit regional formula for \(\normH{(a,b,c)}\), valid for all coefficients and written in closed form on five coefficient regions.  Second, we prove that the projection of \(\BH\) is exactly the square \([-1,1]^2\), compute both endpoints of every vertical section, and obtain a complete parametrization of \(\SH\).  Finally, we determine all extreme points of \(\BH\): they consist of four explicit curves, their negatives, and four pairs of isolated points.  To the best of our knowledge, this is the first complete description of the unit ball geometry for quadratic forms on the nonsymmetric octagonal sector \eqref{eq:defH}.

The proofs combine three ingredients that are characteristic of this area: a reduction of the norm computation to one-dimensional restrictions on the radial boundary, an exact comparison of the competing endpoint and stationary values, and a careful geometric analysis of the projected coefficient regions.  The resulting formulas are sufficiently explicit to be used directly in subsequent applications of the Krein--Milman method.

\begin{figure}[t]
\centering
\begin{tikzpicture}[scale=3.05]
  \pgfmathsetmacro{\tt}{sqrt(2)-1}
  \draw[->] (-0.18,0) -- (1.18,0) node[right] {$x$};
  \draw[->] (0,-0.18) -- (0,1.18) node[above] {$y$};
  \draw[gray!55,dashed,thick]
    (1,\tt)--(\tt,1)--(-\tt,1)--(-1,\tt)--(-1,-\tt)--(-\tt,-1)--(\tt,-1)--(1,-\tt)--cycle;
  \fill[gray!18] (0,0)--(1,0)--(1,\tt)--(\tt,1)--(0,1)--cycle;
  \draw[very thick] (0,0)--(1,0)--(1,\tt)--(\tt,1)--(0,1)--cycle;
  \draw[very thick] (0,1)--(\tt,1);
  \draw[very thick] (\tt,1)--(1,\tt);
  \draw[very thick] (1,\tt)--(1,0);
  \node at (0.47,0.55) {$\HH$};
  \node[above] at (0.2,1.00) {$L_1$};
  \node[above right] at (0.73,0.74) {$L_2$};
  \node[right] at (1.00,0.19) {$L_3$};
  \fill (\tt,1) circle (0.6pt);
  \fill (1,\tt) circle (0.6pt);
  \node[above right,font=\scriptsize,xshift=1pt,yshift=1pt] at (\tt,1) {$(\tau,1)$};
  \node[right,font=\scriptsize] at (1,\tt) {$(1,\tau)$};
  \fill (0,0) circle (0.6pt);
  \fill (1,0) circle (0.6pt);
  \fill (0,1) circle (0.6pt);
  \node[below left,font=\scriptsize] at (0,0) {$(0,0)$};
  \node[below,font=\scriptsize] at (0.88,0) {$(1,0)$};
  \node[left,font=\scriptsize] at (0,0.9) {$(0,1)$};
\end{tikzpicture}
\caption{The nonsymmetric body \(\HH\) and the three radial boundary segments \(L_1\), \(L_2\), and \(L_3\).  The dashed polygon is the centrally symmetric regular octagon whose first-quadrant sector is \(\HH\).}
\label{fig:domain}
\end{figure}
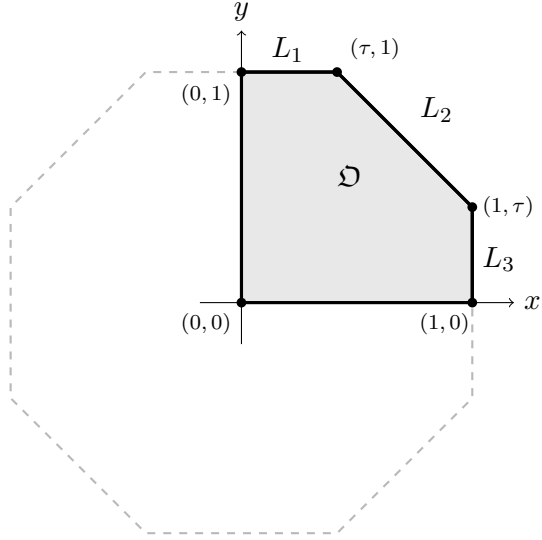

The paper is organized as follows.  In \cref{sec:norm} we prove the explicit norm formula.  In \cref{sec:projection} we determine the projection of \(\BH\), compute the vertical sections, and parametrize \(\SH\).  The complete description of the extreme points is then obtained in \cref{sec:extreme}.

\section{An explicit formula for the norm}\label{sec:norm}

We now derive a regional formula for
\[
 \normH{(a,b,c)}
 =\sup\{|ax^2+bxy+cy^2|:(x,y)\in\HH\}.
\]
As in the works cited above, the case in which the mixed coefficient vanishes is treated first, and the case $b\ne0$ is reduced to a two-dimensional problem by normalizing $b=1$.

\subsection{The case \texorpdfstring{$b=0$}{b=0}}

\begin{proposition}\label{prop:bzero}
For every $a,c\in\R$,
\begin{equation}\label{eq:bzero}
 \normH{(a,0,c)}
 =\max\bigl\{|a|,|c|,|\tau^2a+c|,|a+\tau^2c|\bigr\}.
\end{equation}
\end{proposition}

\begin{proof}
Put $X=x^2$ and $Y=y^2$.  Since the convex function $(x,y)\mapsto x^2+y^2$ attains its maximum over the polygon $\HH$ at a vertex,
\[
 0\le X\le1,\qquad 0\le Y\le1,
 \qquad X+Y\le1+\tau^2.
\]
Consequently,
\[
 \{(x^2,y^2):(x,y)\in\HH\}
 \subset Q:=\operatorname{conv}\{(0,0),(1,0),(1,\tau^2),(\tau^2,1),(0,1)\}.
\]
All five vertices of $Q$ belong to the set on the left.  The map $(X,Y)\mapsto aX+cY$ is affine, so the maximum of its absolute value over $Q$ is attained at one of those vertices.  Formula \eqref{eq:bzero} follows.
\end{proof}

\subsection{Normalization and the five regions}

Assume from now on that $b\ne0$.  By homogeneity in the coefficients,
\begin{equation}\label{eq:normalize-b}
 \normH{(a,b,c)}
 =|b|\,\nu\!\left(\frac ab,\frac cb\right),
 \qquad
 \nu(u,v):=\normH{(u,1,v)}.
\end{equation}
Thus it is enough to determine $\nu$ on the $uv$-plane.

The endpoint values on $L_1$ and $L_3$ are encoded by
\begin{align}
 \Phi_1(u,v)
 &:=\max\{|v|,|\tau^2u+\tau+v|\}\notag\\
 &=\left|v+\frac{\tau+\tau^2u}{2}\right|
   +\frac12|\tau+\tau^2u|,\label{eq:Phi1}\\
 \Phi_3(u,v)
 &:=\max\{|u|,|u+\tau+\tau^2v|\}\notag\\
 &=\left|u+\frac{\tau+\tau^2v}{2}\right|
   +\frac12|\tau+\tau^2v|.\label{eq:Phi3}
\end{align}
We also set
\[
 \eta:=\frac1{2\tau}=\frac{1+\sqrt2}{2},
 \qquad
 \delta:=\frac{1-\tau}{2}=1-\frac1{\sqrt2}.
\]

The three regions associated with an interior contact point are
\begin{align}
 \mathcal I_1
 &:=\left\{(u,v)\in\R^2:
 u\le-\eta,\quad v\ge-u+\frac1{4u}\right\},\label{eq:I1}\\
 \mathcal I_2
 &:=\Bigl\{(u,v)\in\R^2:
 v\le\delta+\tau u,\quad u\le\delta+\tau v,\notag\\[-1mm]
 &\hspace{24mm}2u^2+6uv-2u-1\le0,\notag\\[-1mm]
 &\hspace{24mm}2v^2+6uv-2v-1\le0\Bigr\},\label{eq:I2}\\
 \mathcal I_3
 &:=\left\{(u,v)\in\R^2:
 v\le-\eta,\quad u\ge-v+\frac1{4v}\right\}.
 \label{eq:I3}
\end{align}
Define
\begin{equation}\label{eq:Psi2}
 \Psi_2(u,v):=
 \begin{cases}
 \dfrac{1-4uv}{2(1-u-v)},&(u,v)\ne(\tfrac12,\tfrac12),\\[2mm]
 1,&(u,v)=(\tfrac12,\tfrac12).
 \end{cases}
\end{equation}
The second value is the continuous extension of the first one along $\mathcal I_2$.  Indeed, writing
\[
 u=\frac12-p,\qquad v=\frac12-q,
\]
one has $p,q\ge0$ on $\mathcal I_2$ and, away from $(p,q)=(0,0)$,
\[
 \Psi_2(u,v)=1-\frac{2pq}{p+q}\longrightarrow1
 \qquad\text{as }(p,q)\longrightarrow(0,0).
\]

The remaining points are assigned to one of the two endpoint envelopes.  Put
\begin{align}
 \mathcal E_0
 &:=\R^2\setminus\bigl(
 \operatorname{int}\mathcal I_1\cup
 \operatorname{int}\mathcal I_2\cup
 \operatorname{int}\mathcal I_3\bigr),\notag\\
 \mathcal E_1
 &:=\{(u,v)\in\mathcal E_0:\Phi_1(u,v)\ge\Phi_3(u,v)\},\label{eq:E1}\\
 \mathcal E_3
 &:=\{(u,v)\in\mathcal E_0:\Phi_3(u,v)\ge\Phi_1(u,v)\}.
 \label{eq:E3}
\end{align}
The interiors of $\mathcal I_1,\mathcal I_2,\mathcal I_3,\mathcal E_1,\mathcal E_3$ are pairwise disjoint, and the union of the five regions is $\R^2$.  To see the disjointness of the three interior-contact regions, observe that
\[
 \operatorname{int}\mathcal I_1\subset\{v>1\},\qquad
 \operatorname{int}\mathcal I_3\subset\{u>1\},\qquad
 \operatorname{int}\mathcal I_2\subset\{u<\tfrac12,\ v<\tfrac12\}.
\]
Moreover, $\Phi_1-\Phi_3$ is piecewise affine and no one of its affine branches vanishes identically on an open set; hence its zero set has empty interior, and the interiors of $\mathcal E_1$ and $\mathcal E_3$ are disjoint.  Finally, $\mathcal E_1\cup\mathcal E_3=\mathcal E_0$, so the five regions cover $\R^2$.  Common boundary points have deliberately been included in both adjacent regions; the corresponding formulas below agree there.  The five regions are shown in \cref{fig:norm-regions}.

\begin{figure}[t]
\centering
\includegraphics[width=\textwidth]{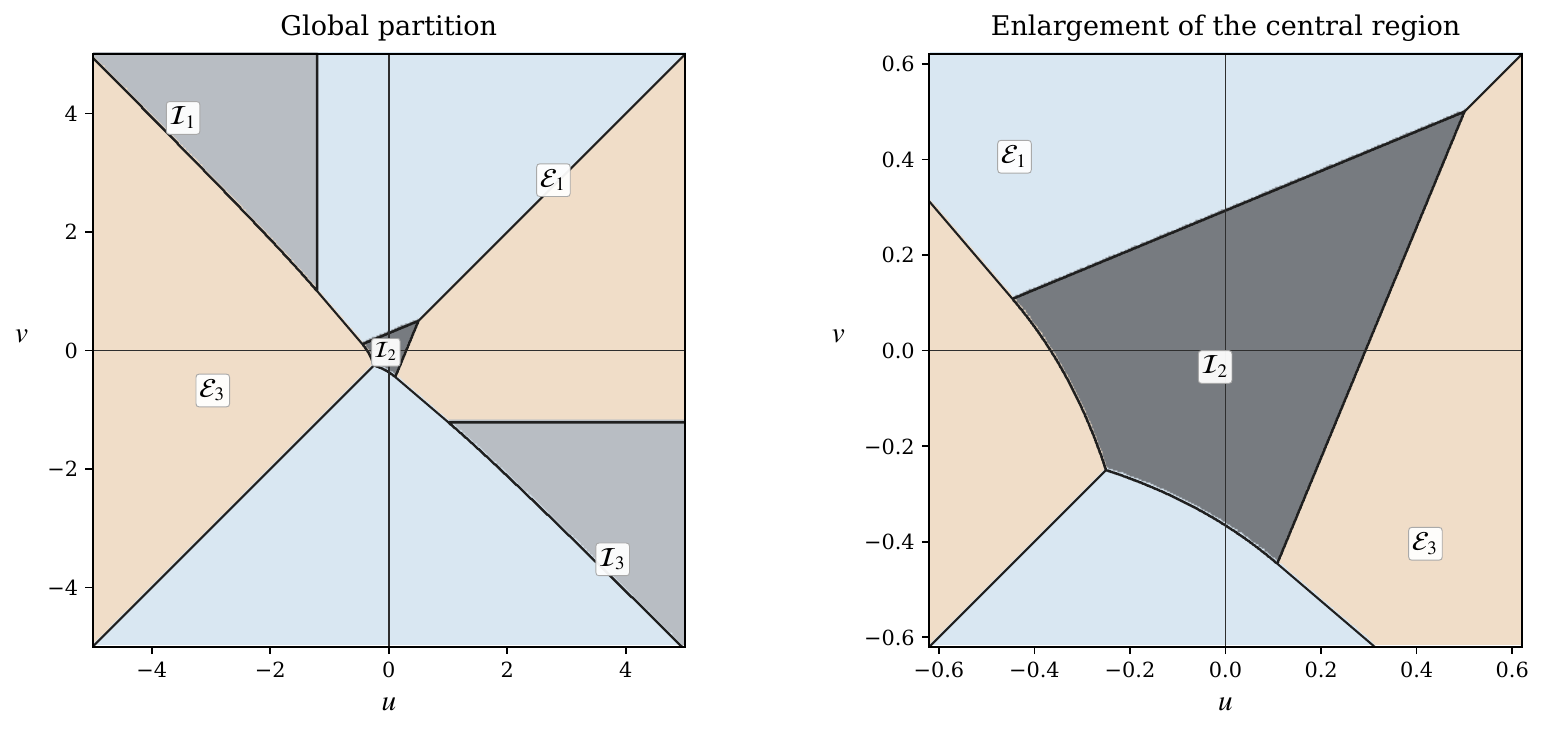}
\caption{The five regions governing the explicit norm formula after the normalization $b=1$.  The regions $\mathcal I_j$ correspond to an interior contact point on $L_j$, while $\mathcal E_1$ and $\mathcal E_3$ correspond to endpoint contact.  To make the two endpoint regions immediately distinguishable, $\mathcal E_1$ is shaded in muted blue and $\mathcal E_3$ in muted ochre; all labels and boundary curves are black.  The right panel magnifies the bounded central region $\mathcal I_2$.}
\label{fig:norm-regions}
\end{figure}

\begin{theorem}[Explicit norm formula]\label{thm:norm}
For every $(u,v)\in\R^2$,
\begin{equation}\label{eq:normalized-norm}
 \nu(u,v)=
 \begin{cases}
 v-\dfrac1{4u},&(u,v)\in\mathcal I_1,\\[2mm]
 \Psi_2(u,v),&(u,v)\in\mathcal I_2,\\[1mm]
 u-\dfrac1{4v},&(u,v)\in\mathcal I_3,\\[2mm]
 \Phi_1(u,v),&(u,v)\in\mathcal E_1,\\[1mm]
 \Phi_3(u,v),&(u,v)\in\mathcal E_3.
 \end{cases}
\end{equation}
Consequently, \eqref{eq:bzero}, \eqref{eq:normalize-b}, and \eqref{eq:normalized-norm} give an explicit formula for $\normH{(a,b,c)}$ for every $(a,b,c)\in\R^3$.
\end{theorem}

\begin{proof}
Let
\[
 Q_{u,v}(x,y)=ux^2+xy+vy^2.
\]
By homogeneity, $|Q_{u,v}|$ attains its maximum on $L_1\cup L_2\cup L_3$.
On $L_1$ and $L_3$ the restrictions are
\[
 q_1(x)=ux^2+x+v\quad(0\le x\le\tau),
 \qquad
 q_3(y)=vy^2+y+u\quad(0\le y\le\tau).
\]
Their endpoint envelopes are precisely $\Phi_1$ and $\Phi_3$.  The only possible stationary candidates are
\[
 v-\frac1{4u}\quad\text{when }u\le-\eta,
 \qquad
 u-\frac1{4v}\quad\text{when }v\le-\eta.
\]
An interior stationary point occurs when the corresponding inequality is strict; equality gives the endpoint $x=\tau$ on $L_1$ or $y=\tau$ on $L_3$.

For $L_2$, write
\[
 (x,y)=\left(\frac{\sqrt2 z}{1+z},\frac{\sqrt2}{1+z}\right),
 \qquad \tau\le z\le\tau^{-1}.
\]
Then
\begin{equation}\label{eq:r-uv}
 r_{u,v}(z):=Q_{u,v}(x,y)
 =\frac{2(uz^2+z+v)}{(1+z)^2},
 \qquad
 r'_{u,v}(z)=
 \frac{2((2u-1)z+1-2v)}{(1+z)^3}.
\end{equation}
Whenever an interior stationary point exists, it is located at
\[
 z_0=\frac{1-2v}{1-2u},
\]
and direct substitution gives
\begin{equation}\label{eq:stationary-L2}
 r_{u,v}(z_0)=\frac{1-4uv}{2(1-u-v)}.
\end{equation}
The endpoint values of $r_{u,v}$ are already contained in $\Phi_1$ and $\Phi_3$.  Thus the only candidates not contained in the endpoint envelopes are the three stationary values just displayed.

We first prove the formula on $\mathcal I_1$.  Put
\[
 \lambda=v-\frac1{4u}.
\]
The defining inequalities of $\mathcal I_1$ imply $u<0$, $\lambda\ge-u$, and $v\ge1$.  Completing the square gives
\begin{equation}\label{eq:I1-square}
 Q_{u,v}(x,y)
 =u\left(x+\frac{y}{2u}\right)^2+\lambda y^2
 \le\lambda.
\end{equation}
On the other hand, because $x,y\ge0$ and $v\ge1$,
\[
 Q_{u,v}(x,y)\ge ux^2\ge u\ge-\lambda.
\]
Equality in the upper estimate is attained at
$(-1/(2u),1)\in L_1$.  Hence $\nu(u,v)=\lambda$ on $\mathcal I_1$.  The proof on $\mathcal I_3$ is symmetric.

We next consider $\mathcal I_2$.  At $(u,v)=(1/2,1/2)$ one has
\[
 Q_{u,v}(x,y)=\frac12(x+y)^2,
\]
so the norm is $1$.  Assume therefore that $(u,v)\ne(1/2,1/2)$ and put
\[
 \lambda=\frac{1-4uv}{2(1-u-v)}.
\]
The two linear inequalities in \eqref{eq:I2} imply $u,v<1/2$; indeed, combining them yields $u,v\le1/2$, and equality in either one forces $u=v=1/2$.  Hence $1-u-v>0$.

The two quadratic inequalities in \eqref{eq:I2} are equivalent to
\begin{equation}\label{eq:lambda-lower}
 \lambda+u\ge0,
 \qquad
 \lambda+v\ge0.
\end{equation}
They also imply that $\lambda>0$.  Indeed, if $\lambda\le0$, then $u,v\ge-\lambda\ge0$, whereas
\[
 1-4uv=2(1-u-v)\lambda\le0.
\]
Thus $uv\ge1/4$, contradicting $u,v<1/2$.  Moreover,
\begin{align}
 1-\lambda
 &=\frac{(1-2u)(1-2v)}{2(1-u-v)}\ge0,\label{eq:lambda-one}\\
 \lambda-2u
 &=\frac{(1-2u)^2}{2(1-u-v)}\ge0,
 \qquad
 \lambda-2v
 =\frac{(1-2v)^2}{2(1-u-v)}\ge0.\label{eq:lambda-two-u}
\end{align}
Finally, the two linear inequalities are equivalent to
\[
 \tau\le z_0=\frac{1-2v}{1-2u}\le\tau^{-1}.
\]
Thus $r_{u,v}$ attains its maximum at $z_0\in[\tau,\tau^{-1}]$, and this maximum is $\lambda$ by \eqref{eq:stationary-L2}.  On the boundary of $\mathcal I_2$, the point $z_0$ may coincide with one of the two endpoints.

For a direct global upper estimate, set
\[
 \alpha=\sqrt{\lambda-2u},
 \qquad
 \gamma=\sqrt{\lambda-2v}.
\]
Equations \eqref{eq:lambda-one}--\eqref{eq:lambda-two-u} give
$\alpha\gamma=1-\lambda$, and therefore
\begin{equation}\label{eq:I2-identity}
 Q_{u,v}(x,y)
 =\frac{\lambda}{2}(x+y)^2
  -\frac12(\alpha x-\gamma y)^2
 \le\lambda
 \qquad ((x,y)\in\HH).
\end{equation}
It remains to verify the lower estimate.  On $L_1$, the minimum of $q_1$ is attained at an endpoint: if $u\ge0$, then $q_1$ is increasing, whereas if $u<0$, it is concave.  The same argument applies on $L_3$.  On $L_2$, the coefficient of $x^2$ in the restriction is $u+v-1<0$, so that restriction is concave and its minimum is again attained at an endpoint.  Hence it is enough to inspect the four radial vertices.  By \eqref{eq:lambda-lower}, $u,v\ge-\lambda$, and
\[
 \tau^2u+\tau+v
 \ge \tau-(1+\tau^2)\lambda\ge-\lambda,
\]
where the last inequality follows from $0<\lambda\le1$; the fourth vertex is symmetric.  Thus $Q_{u,v}\ge-\lambda$ on $\HH$.  Since equality in the upper estimate occurs at the stationary point on $L_2$, we obtain $\nu(u,v)=\lambda$ on $\mathcal I_2$.

It remains to justify that no other stationary value can strictly dominate the endpoint envelopes outside the interiors of $\mathcal I_1,\mathcal I_2,\mathcal I_3$.  For $q_1$, an interior stationary point can occur only when $u\le-\eta$.  If its value is nonpositive, its absolute value is no larger than the absolute value at an endpoint because the stationary point is the maximum of the concave quadratic $q_1$.  If it is positive and strictly larger than all endpoint values, then in particular
\[
 v-\frac1{4u}>|u|=-u,
\]
which is exactly the strict form of the second inequality defining $\mathcal I_1$.  The argument for $q_3$ is symmetric.

For $r_{u,v}$, an interior stationary point that strictly dominates the endpoint values cannot be a local minimum.  Indeed, if $u>1/2$, then the existence of a positive stationary ratio forces $v>1/2$, and $r_{u,v}$ is positive; a local minimum cannot dominate its endpoints.  Thus a strictly dominant stationary point must satisfy $u,v<1/2$ and is a positive local maximum.  The condition $z_0\in(\tau,\tau^{-1})$ gives the strict forms of the two linear inequalities in \eqref{eq:I2}, while strict domination over the axial endpoint values gives
$\lambda>|u|,|v|$, and hence the strict forms of the two quadratic inequalities in \eqref{eq:I2}.  Therefore such a point belongs to $\operatorname{int}\mathcal I_2$.

Consequently, on $\mathcal E_0$ the norm is simply
\[
 \max\{\Phi_1(u,v),\Phi_3(u,v)\}.
\]
The definitions of $\mathcal E_1$ and $\mathcal E_3$ now yield the last two branches of \eqref{eq:normalized-norm}.  The formulas coincide on every common boundary, which completes the proof.
\end{proof}

\begin{corollary}[Explicit formula in the original coefficients]\label{cor:norm-original-coefficients}
Define
\begin{align*}
 \Theta_1(a,b,c)&:=\max\bigl\{|c|,|\tau^2a+\tau b+c|\bigr\},\\
 \Theta_3(a,b,c)&:=\max\bigl\{|a|,|a+\tau b+\tau^2c|\bigr\},
\end{align*}
and
\begin{align*}
 \Xi_1(a,b,c)&:=\operatorname{sgn}(b)\left(c-\frac{b^2}{4a}\right),\\
 \Xi_2(a,b,c)&:=\operatorname{sgn}(b)\,\frac{4ac-b^2}{2(a-b+c)},\\
 \Xi_3(a,b,c)&:=\operatorname{sgn}(b)\left(a-\frac{b^2}{4c}\right),
\end{align*}
with the continuous extension \(\Xi_2(b/2,b,b/2)=|b|\).
Then, for every \((a,b,c)\in\R^3\),
\[
 \normH{(a,b,c)}=
 \begin{cases}
 \max\bigl\{|a|,|c|,|\tau^2a+c|,|a+\tau^2c|\bigr\},& b=0,\\[1.2ex]
 \Xi_1(a,b,c),& b\ne0\ \text{and}\ \left(\dfrac ab,\dfrac cb\right)\in\mathcal I_1,\\[2ex]
 \Xi_2(a,b,c),& b\ne0\ \text{and}\ \left(\dfrac ab,\dfrac cb\right)\in\mathcal I_2,\\[2ex]
 \Xi_3(a,b,c),& b\ne0\ \text{and}\ \left(\dfrac ab,\dfrac cb\right)\in\mathcal I_3,\\[2ex]
 \Theta_1(a,b,c),& b\ne0\ \text{and}\ \left(\dfrac ab,\dfrac cb\right)\in\mathcal E_1,\\[2ex]
 \Theta_3(a,b,c),& b\ne0\ \text{and}\ \left(\dfrac ab,\dfrac cb\right)\in\mathcal E_3.
 \end{cases}
\]
In particular, the two endpoint branches are more naturally written as the maxima \(\Theta_1\) and \(\Theta_3\), rather than in the equivalent half-sum form \eqref{eq:Phi1}--\eqref{eq:Phi3}.
\end{corollary}

\begin{proof}
The case $b=0$ is precisely \cref{prop:bzero}.  Assume that $b\ne0$ and put
\[
 u=\frac ab,\qquad v=\frac cb.
\]
By \eqref{eq:normalize-b} and \cref{thm:norm},
\[
 \normH{(a,b,c)}=|b|\,\nu(u,v).
\]
On $\mathcal I_1$,
\[
 |b|\left(v-\frac1{4u}\right)
 =\operatorname{sgn}(b)\left(c-\frac{b^2}{4a}\right)
 =\Xi_1(a,b,c).
\]
Similarly, the branch on $\mathcal I_3$ gives $\Xi_3(a,b,c)$.  On $\mathcal I_2$,
\[
 |b|\frac{1-4uv}{2(1-u-v)}
 =\operatorname{sgn}(b)\frac{4ac-b^2}{2(a-b+c)}
 =\Xi_2(a,b,c).
\]
The denominator in the last expression can vanish on $\mathcal I_2$ only when $u+v=1$.  Since $u,v\le1/2$ on $\mathcal I_2$, this happens only at $u=v=1/2$, or equivalently at $a=c=b/2$; there the continuous extension equals $|b|$.  Finally,
\[
 |b|\Phi_1(u,v)=\Theta_1(a,b,c),
 \qquad
 |b|\Phi_3(u,v)=\Theta_3(a,b,c).
\]
This proves the stated formula.
\end{proof}

\begin{remark}\label{rem:norm}
The regional formula is invariant under the symmetry $(u,v)\mapsto(v,u)$, which interchanges $\mathcal I_1$ with $\mathcal I_3$ and $\mathcal E_1$ with $\mathcal E_3$, while leaving $\mathcal I_2$ invariant.  Geometrically, this is the symmetry $x\leftrightarrow y$ of the body $\HH$.
\end{remark}

\section{Projection of the unit ball and parametrization of the unit sphere}\label{sec:projection}

For fixed \((a,c)\), the coefficient \(b\) occurs linearly.  This makes the vertical sections over the \(ac\)-plane more transparent than the direct normalization \(b=1\).

Suppose first that \((a,c)\in[-1,1]^2\).  The upper inequality
\[
 ax^2+bxy+cy^2\le1
\]
can be written, when \(xy>0\), as
\[
 b\le \frac{1-ax^2-cy^2}{xy}.
\]
Consequently, the largest possible value of \(b\) compatible with the upper inequality is
\begin{equation}\label{eq:defU}
 U(a,c):=\inf_{(x,y)\in L_1\cup L_2\cup L_3,\ xy>0}
 \frac{1-ax^2-cy^2}{xy}.
\end{equation}
On the three sides this becomes
\begin{align}
 U_1(a,c)&=\inf_{0<x\le\tau}\left(\frac{1-c}{x}-ax\right),\label{eq:U1}\\
 U_3(a,c)&=\inf_{0<y\le\tau}\left(\frac{1-a}{y}-cy\right),\label{eq:U3}
\end{align}
and, after setting \(z=x/y\) on \(L_2\),
\begin{equation}\label{eq:U2}
 U_2(a,c)=\inf_{\tau\le z\le\tau^{-1}}
 \left[1+\frac12\left((1-2a)z+\frac{1-2c}{z}\right)\right].
\end{equation}
Thus \(U=\min\{U_1,U_2,U_3\}\).

We now introduce the five regions that determine which contact point is active:
\begin{align*}
 \mathcal A&=\{(a,c)\in[-1,1]^2:a\le0,\ 1+\tau^2a\le c\le1\},\\
 \mathcal B&=\{(a,c)\in[-1,1]^2:a\le c,\ \tau+\tau^2a\le c\le1+\tau^2a\},\\
 \mathcal C&=\{(a,c)\in[-1,1]^2:c\le\tau+\tau^2a,\ a\le\tau+\tau^2c\},\\
 \mathcal D&=\{(a,c)\in[-1,1]^2:c\le a,\ \tau+\tau^2c\le a\le1+\tau^2c\},\\
 \mathcal E&=\{(a,c)\in[-1,1]^2:c\le0,\ 1+\tau^2c\le a\le1\}.
\end{align*}
Their interiors are pairwise disjoint and
\begin{equation}\label{eq:partition}
 [-1,1]^2=\mathcal A\cup\mathcal B\cup\mathcal C\cup\mathcal D\cup\mathcal E.
\end{equation}
The overlaps consist only of boundary pieces.  The partition is shown in \cref{fig:projection-regions}.

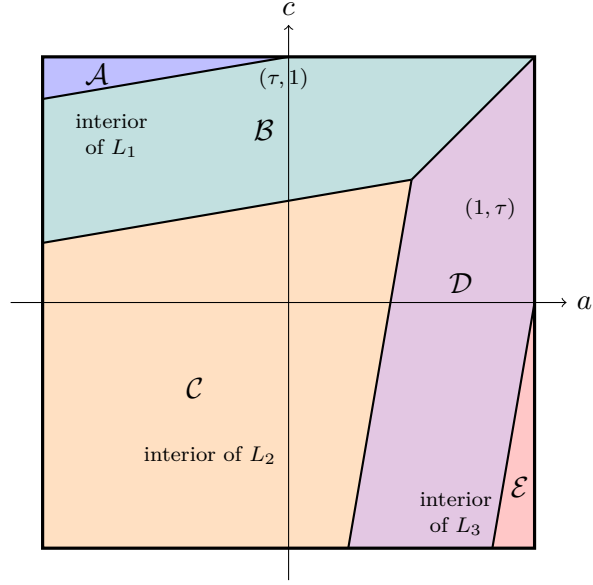
\begin{figure}[t]
\centering
\begin{tikzpicture}[scale=3.25]
  \pgfmathsetmacro{\ttau}{sqrt(2)-1}
  \pgfmathsetmacro{\ttwo}{2*(sqrt(2)-1)}
  \pgfmathsetmacro{\trzero}{(sqrt(2)-1)-(3-2*sqrt(2))}
  \fill[blue!25] (-1,\ttwo)--(-1,1)--(0,1)--cycle;
  \fill[teal!24] (-1,\trzero)--(-1,\ttwo)--(0,1)--(1,1)--(.5,.5)--cycle;
  \fill[orange!24] (-1,-1)--(-1,\trzero)--(.5,.5)--(\trzero,-1)--cycle;
  \fill[violet!22] (\trzero,-1)--(\ttwo,-1)--(1,0)--(1,1)--(.5,.5)--cycle;
  \fill[red!22] (\ttwo,-1)--(1,-1)--(1,0)--cycle;
  \draw[very thick] (-1,-1) rectangle (1,1);
  \draw[thick] (-1,\ttwo)--(0,1);
  \draw[thick] (-1,\trzero)--(.5,.5);
  \draw[thick] (.5,.5)--(\trzero,-1);
  \draw[thick] (\ttwo,-1)--(1,0);
  \draw[thick] (.5,.5)--(1,1);
  \node at (-.78,.93) {$\mathcal A$};
  \node at (-.1,.70) {$\mathcal B$};
  \node at (-.38,-.35) {$\mathcal C$};
  \node at (.70,.08) {$\mathcal D$};
  \node at (.94,-.75) {$\mathcal E$};
  \draw[->] (-1.13,0)--(1.13,0) node[right] {$a$};
  \draw[->] (0,-1.13)--(0,1.13) node[above] {$c$};
  \node[align=center,font=\scriptsize] at (-.72,.68) {interior\\of $L_1$};
  \node[align=center,font=\scriptsize] at (-.02,.91) {$(\tau,1)$};
  \node[align=center,font=\scriptsize] at (-.32,-.62) {interior of $L_2$};
  \node[align=center,font=\scriptsize] at (.82,.38) {$(1,\tau)$};
  \node[align=center,font=\scriptsize] at (.68,-.86) {interior\\of $L_3$};
\end{tikzpicture}
\caption{The five regions in the projected square.  They record whether the upper contact occurs at an interior point of \(L_1\), at \((\tau,1)\), at an interior point of \(L_2\), at \((1,\tau)\), or at an interior point of \(L_3\), respectively.}
\label{fig:projection-regions}
\end{figure}

Define \(F:[-1,1]^2\to\R\) by
\begin{equation}\label{eq:F}
 F(a,c)=
 \begin{cases}
 2\sqrt{-a(1-c)},&(a,c)\in\mathcal A,\\[.8ex]
 \dfrac{1-c}{\tau}-\tau a,&(a,c)\in\mathcal B,\\[1.2ex]
 1+\sqrt{(1-2a)(1-2c)},&(a,c)\in\mathcal C,\\[.8ex]
 \dfrac{1-a}{\tau}-\tau c,&(a,c)\in\mathcal D,\\[1.2ex]
 2\sqrt{-c(1-a)},&(a,c)\in\mathcal E.
 \end{cases}
\end{equation}
The formulas agree on all overlaps.  For example, on \(c=1+\tau^2a\),
\[
 2\sqrt{-a(1-c)}=-2\tau a=\frac{1-c}{\tau}-\tau a,
\]
and on \(c=\tau+\tau^2a\),
\[
 \frac{1-c}{\tau}-\tau a
 =1+\sqrt{(1-2a)(1-2c)}.
\]
The remaining identities follow by symmetry.

\begin{theorem}[Projection and upper endpoint]\label{thm:projection}
The projection of both \(\BH\) and \(\SH\) onto the \(ac\)-plane is the square
\[
 \pi_{ac}(\BH)=\pi_{ac}(\SH)=[-1,1]^2.
\]
Moreover, \(F(a,c)=U(a,c)\) for every \((a,c)\in[-1,1]^2\).  Equivalently, \(F(a,c)\) is the largest coefficient \(b\) for which
\[
 ax^2+bxy+cy^2\le1\qquad ((x,y)\in\HH).
\]
For each \((a,c)\in[-1,1]^2\), the polynomial \((a,F(a,c),c)\) belongs to \(\SH\).
\end{theorem}

\begin{proof}
The points \((1,0)\) and \((0,1)\) belong to \(\HH\).  Hence
\[
 \normH{(a,b,c)}\ge\max\{|a|,|c|\},
\]
which proves that both projections are contained in \([-1,1]^2\).

We compute the upper endpoint.  If \(a<0\) and \(1+\tau^2a<c<1\), the minimum in \eqref{eq:U1} is attained at the interior point
\[
 x_0=\sqrt{\frac{1-c}{-a}}
\]
and equals \(2\sqrt{-a(1-c)}\).  The same formula holds on the boundary of \(\mathcal A\) by continuity: at \(x=\tau\) when \(c=1+\tau^2a\), and as \(x\downarrow0\) when \(c=1\).  In all remaining cases the minimum is attained at \(x=\tau\) and equals
\[
 u_{\tau,1}(a,c):=\frac{1-c}{\tau}-\tau a.
\]
The calculation of \(U_3\) is symmetric.

For \(U_2\), the stationary ratio is
\[
 z_0=\sqrt{\frac{1-2c}{1-2a}}.
\]
The inequalities defining \(\mathcal C\) are equivalent to
\[
 1-2a\ge0,\qquad 1-2c\ge0,\qquad
 \tau^2\le\frac{1-2c}{1-2a}\le\tau^{-2},
\]
with the usual limiting interpretation at \((a,c)=(1/2,1/2)\).  Thus, on \(\mathcal C\), the minimum in \eqref{eq:U2} is attained at \(z_0\) and equals
\[
 1+\sqrt{(1-2a)(1-2c)}.
\]
At the endpoints of the interval for \(z\) one obtains
\[
 u_{\tau,1}(a,c),\qquad
 u_{1,\tau}(a,c):=\frac{1-a}{\tau}-\tau c,
\]
and
\begin{equation}\label{eq:endpointdifference}
 u_{\tau,1}(a,c)-u_{1,\tau}(a,c)=2(a-c).
\end{equation}
On \(\mathcal B\), the inequality \(c\le1+\tau^2a\) makes the minimum in \eqref{eq:U1} occur at \(x=\tau\), while \(c\ge\tau+\tau^2a\) is exactly the condition that the minimum in \eqref{eq:U2} occur at the endpoint \(z=\tau\), with the usual boundary convention.  Finally, \(a\le c\) and \eqref{eq:endpointdifference} select \((\tau,1)\) rather than \((1,\tau)\).  The region \(\mathcal D\) is symmetric.  The regions \(\mathcal A\) and \(\mathcal E\) correspond to the two remaining interior minima.  These alternatives are exhaustive: the infimum defining $U_1$ is obtained either at its stationary point or at $x=\tau$ (with the limiting case $x\downarrow0$ when $c=1$), the same dichotomy holds symmetrically for $U_3$, and the infimum defining $U_2$ is obtained either at its stationary point or at one of the two endpoints of $[\tau,\tau^{-1}]$.  Together with \eqref{eq:endpointdifference}, these possibilities give exactly the five regions in \eqref{eq:partition}.  This proves \(U=F\) and also \eqref{eq:partition}.

It remains to verify the lower inequality for \(b=F(a,c)\).  We do this region by region.

If \((a,c)\in\mathcal A\), then \(c\ge1-\tau^2=2\tau>0\) and \(F(a,c)\ge0\).  Therefore
\[
 ax^2+F(a,c)xy+cy^2\ge ax^2\ge-1
\]
throughout \(\HH\).  The case \(\mathcal E\) is symmetric.

Let \((a,c)\in\mathcal B\), and put \(b=F(a,c)\).  Here
\(c\ge\tau-\tau^2>0\).  If \(b\ge0\), then again
\(P(x,y)\ge ax^2\ge-1\).  If \(b<0\), set
\[
 d=c-1+\tau^2a>0.
\]
Then \(b=-d/\tau\).  The upper boundary of \(\mathcal B\) gives
\(d\le2\tau^2a\), so \(a>0\), whereas the lower boundary gives
\(c\ge\tau+\tau^2a>\tau^2a\).  Consequently,
\[
 b^2=\frac{d^2}{\tau^2}\le4\tau^2a^2\le4ac.
\]
Thus the quadratic form is positive semidefinite, and in particular it is bounded below by \(0\).  The case \(\mathcal D\) follows by symmetry.

Finally, let \((a,c)\in\mathcal C\), put
\[
 \alpha=\sqrt{1-2a},\qquad \gamma=\sqrt{1-2c},
\]
and set \(b=1+\alpha\gamma\).  Then
\begin{equation}\label{eq:Cidentity}
 P(x,y)=\frac12(x+y)^2-\frac12(\alpha x-\gamma y)^2.
\end{equation}
The inequalities defining \(\mathcal C\) imply
\(\tau\alpha\le\gamma\le\tau^{-1}\alpha\), and also \(a+c\le1\).  On \(L_1\), the restriction is increasing when \(a\ge0\), and concave when \(a<0\); hence its minimum occurs at an endpoint.  The same holds on \(L_3\).  On \(L_2\), the leading coefficient is
\(a-b+c\le0\), so the restriction is concave and its minimum also occurs at an endpoint.  At the two non-axial endpoints, \eqref{eq:Cidentity} gives
\begin{align*}
 P(\tau,1)&=1-\frac12(\alpha\tau-\gamma)^2,\\
 P(1,\tau)&=1-\frac12(\alpha-\gamma\tau)^2.
\end{align*}
Indeed, $\tau\alpha\le\gamma$ and $\tau\gamma\le\alpha$, so
$|\alpha\tau-\gamma|\le\gamma\le\sqrt3$ and
$|\alpha-\gamma\tau|\le\alpha\le\sqrt3$.  Hence both squared terms are at most $3$, and therefore both endpoint values are at least $-1/2$.  The axial endpoint values satisfy \(a,c\ge-1\).  Therefore \(P\ge-1\) on the radial boundary and hence on all of \(\HH\).

We have shown in every case that \(\normH{(a,F(a,c),c)}\le1\).  Equality holds at the contact point used to obtain \(F\): at an interior point of \(L_1\) in \(\mathcal A\), at \((\tau,1)\) in \(\mathcal B\), at an interior point of \(L_2\) in \(\mathcal C\), at \((1,\tau)\) in \(\mathcal D\), and at an interior point of \(L_3\) in \(\mathcal E\).  The limiting boundary cases are immediate.  Hence \((a,F(a,c),c)\in\SH\), proving the reverse inclusion of the projected square and completing the proof.
\end{proof}

Define
\begin{equation}\label{eq:G}
 G(a,c):=-F(-a,-c).
\end{equation}
The lower inequality \(ax^2+bxy+cy^2\ge-1\) is equivalent to the upper inequality for \((-a,-b,-c)\).  We therefore obtain the complete vertical description of the ball.

\begin{corollary}[Vertical sections]\label{cor:ball}
For every \((a,b,c)\in\R^3\),
\[
 (a,b,c)\in\BH
 \quad\Longleftrightarrow\quad
 (a,c)\in[-1,1]^2\ \text{ and }\ G(a,c)\le b\le F(a,c).
\]
In particular, \(F\) is concave, \(G\) is convex, and both are continuous.
\end{corollary}

\begin{proof}
The preceding theorem identifies the upper endpoint.  Applying it to \((-a,-c)\) identifies the lower endpoint as \(-F(-a,-c)=G(a,c)\).  The concavity of \(F\) also follows directly from \eqref{eq:defU}: it is the infimum of affine functions of \((a,c)\).  Continuity follows from the explicit formula \eqref{eq:F}: every branch is continuous on its region and the branches agree on all common boundaries.  The continuity of \(G\) then follows from \eqref{eq:G}.
\end{proof}

Let
\[
 \mathcal V=\{(a,b,c):(a,c)\in\partial[-1,1]^2,\ G(a,c)\le b\le F(a,c)\}.
\]

\begin{theorem}[Unit sphere]\label{thm:sphere}
The unit sphere is
\begin{equation}\label{eq:sphere}
 \SH=\graph(F)\cup\graph(G)\cup\mathcal V.
\end{equation}
\end{theorem}

\begin{proof}
The description of \(\BH\) in \cref{cor:ball} shows that its boundary contains the upper and lower graphs.  It also contains every vertical section over the boundary of the projected square.  Conversely, if \((a,c)\) is an interior point of the square and
\(G(a,c)<b<F(a,c)\), continuity gives a neighbourhood contained in \(\BH\); hence the point is not on the sphere.  This proves \eqref{eq:sphere}.
\end{proof}

\begin{figure}[t]
\centering
\includegraphics[width=.75\textwidth]{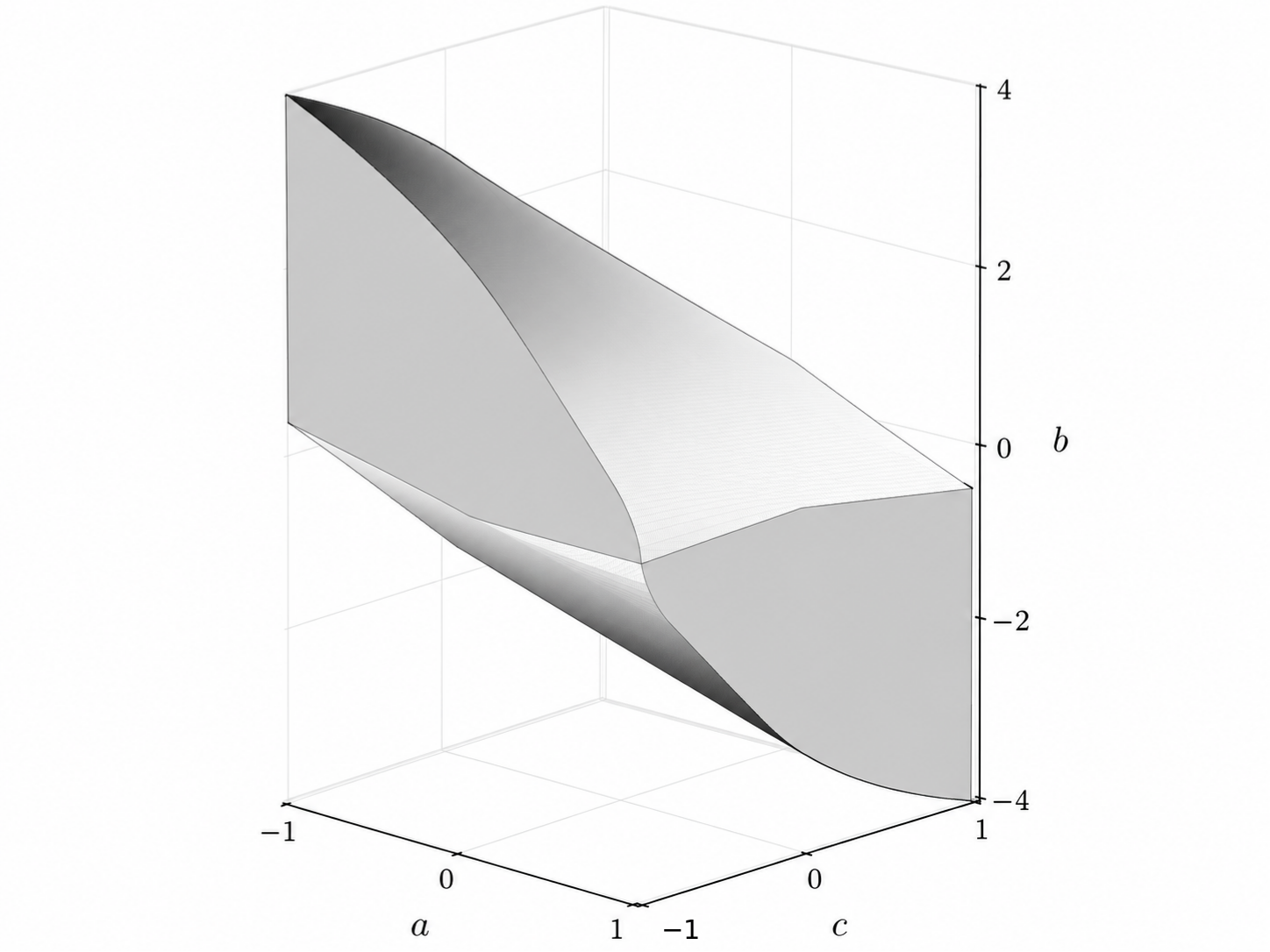}
\caption{Representation of the unit sphere $\SH$ associated with the space of quadratic forms $P(x,y)=ax^2+bxy+cy^2$ endowed with the supremum norm on the octagonal sector
$\HH=\{(x,y)\in[0,1]^2:x+y\le\sqrt2\}$, with $b$ as the vertical coordinate.  The extreme points are not superimposed on the sphere and are displayed separately in \cref{fig:extreme-set}.}
\label{fig:sphere}
\end{figure}

\begin{remark}
The norm is not absolute.  Indeed, \((1,-\tau,1)\in\SH\), whereas
\[
 \normH{(1,\tau,1)}\ge P(\tau,1)=1+2\tau^2>1.
\]
\end{remark}

\section{Extreme points of the unit ball}\label{sec:extreme}

The vertical description in \cref{cor:ball} allows us to determine the extreme points without relying on a three-dimensional plot.  We first record the precise convex-geometric criterion that will be used.

\begin{lemma}\label{lem:criterion}
Let \(D\subset\R^n\) be convex, let \(F:D\to\R\) be concave and \(G:D\to\R\) convex, with \(G\le F\), and set
\[
 K=\{(x,t)\in D\times\R:G(x)\le t\le F(x)\}.
\]
Then \((x_0,F(x_0))\) is an extreme point of \(K\) if and only if there is no nondegenerate segment \([x_1,x_2]\subset D\) such that \(x_0\) belongs to its relative interior and \(F\) is affine on \([x_1,x_2]\).
\end{lemma}

\begin{proof}
If \(F\) is affine on such a segment, the graph point over \(x_0\) is a nontrivial convex combination of the graph points over \(x_1\) and \(x_2\).  Conversely, suppose that
\[
 (x_0,F(x_0))=\lambda(x_1,t_1)+(1-\lambda)(x_2,t_2),
 \qquad 0<\lambda<1,
\]
with \((x_i,t_i)\in K\).  Concavity gives
\[
 F(x_0)=\lambda t_1+(1-\lambda)t_2
 \le \lambda F(x_1)+(1-\lambda)F(x_2)
 \le F(x_0).
\]
Hence equality holds throughout, so \(t_i=F(x_i)\) and equality holds in the concavity inequality.  A concave function satisfying equality at an interior point of a segment is affine on the whole segment.
\end{proof}

Put
\[
 \sigma:=\tau-\tau^2=3\sqrt2-4,
\]
and define the following four curves, where the coordinates are ordered as \((a,b,c)\):
\begin{align*}
 \Gamma_1&=\{(-1,2\sqrt{1-t},t):2\tau\le t\le1\},\\
 \Gamma_2&=\{(-1,1+\sqrt{3(1-2t)},t):-1\le t\le\sigma\},\\
 \Gamma_3&=\{(t,1+\sqrt{3(1-2t)},-1):-1\le t\le\sigma\},\\
 \Gamma_4&=\{(t,2\sqrt{1-t},-1):2\tau\le t\le1\}.
\end{align*}
We also set
\[
 p_0=(0,0,1),\qquad
 p_1=(1,0,0),\qquad
 p_2=\left(\frac12,1,\frac12\right),\qquad
 p_3=(1,-\tau,1).
\]

\begin{theorem}[Extreme points]\label{thm:extreme}
The extreme points of \(\BH\) are
\begin{equation}\label{eq:extreme}
 \ext(\BH)
 =\bigcup_{j=1}^4\bigl(\Gamma_j\cup(-\Gamma_j)\bigr)
 \cup\{\pm p_0,\pm p_1,\pm p_2,\pm p_3\}.
\end{equation}
\end{theorem}

\begin{proof}
We divide the argument into the exclusion of all other points and the verification that every point listed in \eqref{eq:extreme} is indeed extreme.

\smallskip
\noindent\emph{Step 1: localization of all possible extreme points.}
Every point in the relative interior of a nondegenerate vertical section of \(\BH\) is the midpoint of two distinct points of that section and is therefore non-extreme.  By central symmetry it is consequently enough to analyze the upper graph \(\graph(F)\).

The three nonlinear pieces of \(F\) are ruled by explicit affine generators.  For $\mathcal A$, set $\lambda=(1-c)/(-a)$ when $a<0$; the case $a=0$ reduces to the common point $(0,1)$.  The inequalities defining $\mathcal A$ are then equivalent to $0\le\lambda\le\tau^2$ and $c=1+\lambda a$.  Hence
\begin{equation}\label{eq:A-generators}
 \mathcal A
 =\bigcup_{0\le\lambda\le\tau^2}
 \{(a,1+\lambda a):-1\le a\le0\},
 \qquad
 F(a,1+\lambda a)=-2\sqrt\lambda\,a.
\end{equation}
Thus every point of \(\graph(F|_{\mathcal A})\), except for the common endpoint \(p_0\) and the endpoints lying on \(a=-1\), belongs to the relative interior of an affine segment.  The latter endpoints form \(\Gamma_1\).

Likewise, setting $\lambda=(1-a)/(-c)$ when $c<0$ gives
\begin{equation}\label{eq:E-generators}
 \mathcal E
 =\bigcup_{0\le\lambda\le\tau^2}
 \{(1+\lambda c,c):-1\le c\le0\},
 \qquad
 F(1+\lambda c,c)=-2\sqrt\lambda\,c.
\end{equation}
The only possible extreme points contributed by this piece are \(p_1\) and the curve \(\Gamma_4\).

For the central region \(\mathcal C\), write \(u=1-2a\) and, when $u>0$, set $\lambda=(1-2c)/u$; the case $u=0$ gives the common point $(1/2,1/2)$.  The defining inequalities of $\mathcal C$ give $\tau^2\le\lambda\le\tau^{-2}$.  Moreover, $a,c\in[-1,1]$ is equivalent to $0\le u\le3$ and $0\le\lambda u\le3$, that is, $0\le u\le3\min\{1,\lambda^{-1}\}$.  Therefore
\begin{equation}\label{eq:C-generators}
 \mathcal C
 =\bigcup_{\tau^2\le\lambda\le\tau^{-2}}
 \left\{
 \left(\frac{1-u}{2},\frac{1-\lambda u}{2}\right):
 0\le u\le3\min\{1,\lambda^{-1}\}
 \right\},
\end{equation}
and on every such segment
\[
 F\left(\frac{1-u}{2},\frac{1-\lambda u}{2}\right)
 =1+\sqrt\lambda\,u.
\]
All these generators have the common endpoint corresponding to \(p_2\).  Their other endpoints lie on \(a=-1\) when \(\lambda\le1\) and on \(c=-1\) when \(\lambda\ge1\); they form \(\Gamma_2\cup\Gamma_3\).

The remaining pieces \(\mathcal B\) and \(\mathcal D\) are planar.  Each is the intersection of the half-planes appearing in its definition with the square $[-1,1]^2$.  Taking the successive intersections of the corresponding boundary lines gives
\begin{align*}
 \mathcal B&=\operatorname{conv}\left\{
 (-1,2\tau),(0,1),(1,1),
 \left(\frac12,\frac12\right),(-1,\sigma)
 \right\},\\
 \mathcal D&=\operatorname{conv}\left\{
 (2\tau,-1),(1,0),(1,1),
 \left(\frac12,\frac12\right),(\sigma,-1)
 \right\}.
\end{align*}
Every nonvertex point of either polygon lies in the relative interior of a segment on which \(F\) is affine.  All their vertices have already appeared in the four curves or among \(p_0,p_1,p_2,p_3\).  By \cref{lem:criterion}, no other point of the upper graph can be extreme.

\smallskip
\noindent\emph{Step 2: the four curves are extreme.}
On the boundary edge \(a=-1\), the upper endpoint has the exact profile
\begin{equation}\label{eq:edgeF}
 F(-1,t)=
 \begin{cases}
 1+\sqrt{3(1-2t)},&-1\le t\le\sigma,\\[1mm]
 \dfrac{1-t}{\tau}+\tau,&\sigma\le t\le2\tau,\\[2mm]
 2\sqrt{1-t},&2\tau\le t\le1.
 \end{cases}
\end{equation}
The first and third branches are strictly concave, while the middle branch is affine.  Indeed,
\begin{align*}
 \frac{d^2}{dt^2}\left(1+\sqrt{3(1-2t)}\right)
 &=-\frac{\sqrt3}{(1-2t)^{3/2}}<0,\\
 \frac{d^2}{dt^2}\left(2\sqrt{1-t}\right)
 &=-\frac1{2(1-t)^{3/2}}<0.
\end{align*}
These inequalities hold on the interiors of the respective parameter intervals.  Since \(a=-1\) is an edge of the projected square, every segment in \([-1,1]^2\) having an interior point with first coordinate \(-1\) must be contained in that edge.  Therefore the only non-extreme upper graph points above \(a=-1\) are those lying over the relative interior of the affine interval \([\sigma,2\tau]\).  The remaining points are precisely \(\Gamma_2\cup\Gamma_1\), including the two junction points.  By the symmetry \((a,c)\mapsto(c,a)\), the same argument proves that every point of \(\Gamma_3\cup\Gamma_4\) is extreme.  The distinction between the curved extreme portions and the affine gap is shown in \cref{fig:edge-profile}.

\begin{figure}[t]
\centering
\includegraphics[width=.86\textwidth]{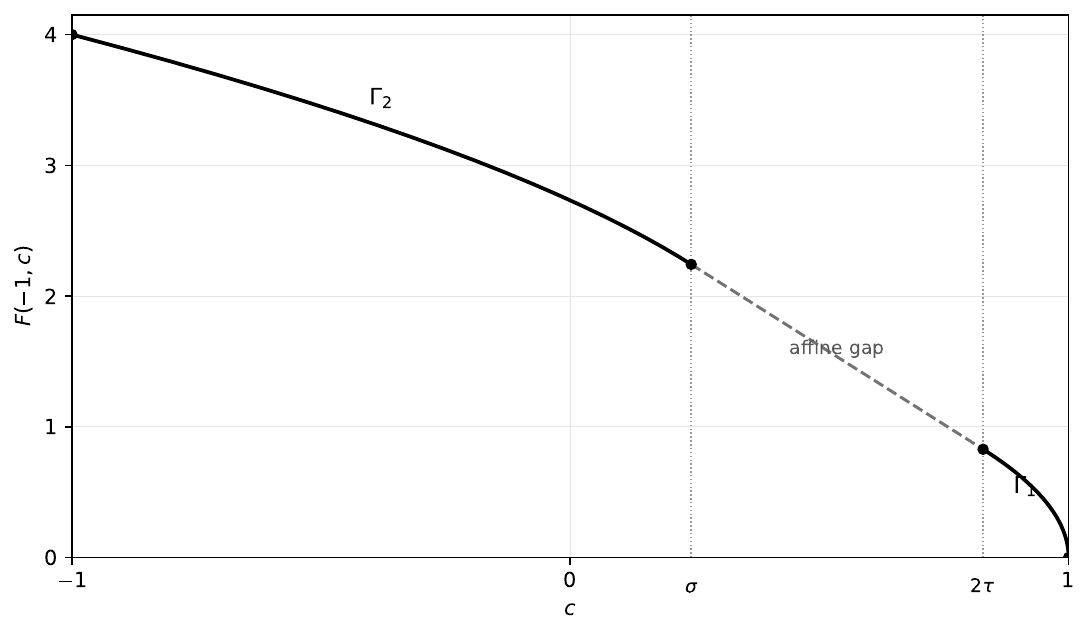}
\caption{The upper profile above the edge \(a=-1\).  The two solid curved portions give \(\Gamma_2\) and \(\Gamma_1\); every point of these portions is extreme.  The dashed middle portion is affine, so its relative interior contains no extreme points.}
\label{fig:edge-profile}
\end{figure}

\smallskip
\noindent\emph{Step 3: the four isolated points are extreme.}
We use a contact-set argument.  Suppose that a polynomial \(P\in\SH\) can be written as
\(P=(P_++P_-)/2\) with \(P_+,P_-\in\BH\), and put \(Q=(P_+-P_-)/2\).  At every point \(z\in\HH\) for which \(|P(z)|=1\), the two inequalities \(|P(z)\pm Q(z)|\le1\) force \(Q(z)=0\).

The polynomial associated with \(p_0\) is \(P_0(x,y)=y^2\), and \(P_0=1\) on the whole segment \(L_1\).  Thus \(Q(x,1)=0\) for every \(x\in[0,\tau]\), which forces the quadratic polynomial \(Q(x,1)\) to vanish identically and hence \(Q=0\).  Therefore \(p_0\) is extreme.  The same argument on \(L_3\) proves that \(p_1\), corresponding to \(P_1(x,y)=x^2\), is extreme.

The polynomial associated with \(p_2\) is
\[
 P_2(x,y)=\frac12(x+y)^2,
\]
and it equals \(1\) on all of \(L_2\).  Hence
\(Q(x,\sqrt2-x)=0\) for every \(x\in[\tau,1]\).  Comparing the three coefficients in this identically zero quadratic polynomial gives \(Q=0\), so \(p_2\) is extreme.

Finally, the polynomial associated with \(p_3\) is
\[
 P_3(x,y)=x^2-\tau xy+y^2.
\]
It takes the value \(1\) at \((1,0)\), \((0,1)\), and \((\tau,1)\).  If
\(Q(x,y)=\alpha x^2+\beta xy+\gamma y^2\), the contact-set observation yields
\[
 \alpha=Q(1,0)=0,\qquad
 \gamma=Q(0,1)=0,\qquad
 \tau\beta=Q(\tau,1)=0.
\]
Thus \(Q=0\), and \(p_3\) is extreme.

We have proved that every upper candidate is extreme and that no other upper graph point is extreme.  Applying central symmetry gives exactly \eqref{eq:extreme}.
\end{proof}

\begin{figure}[t]
\centering
\includegraphics[width=\textwidth]{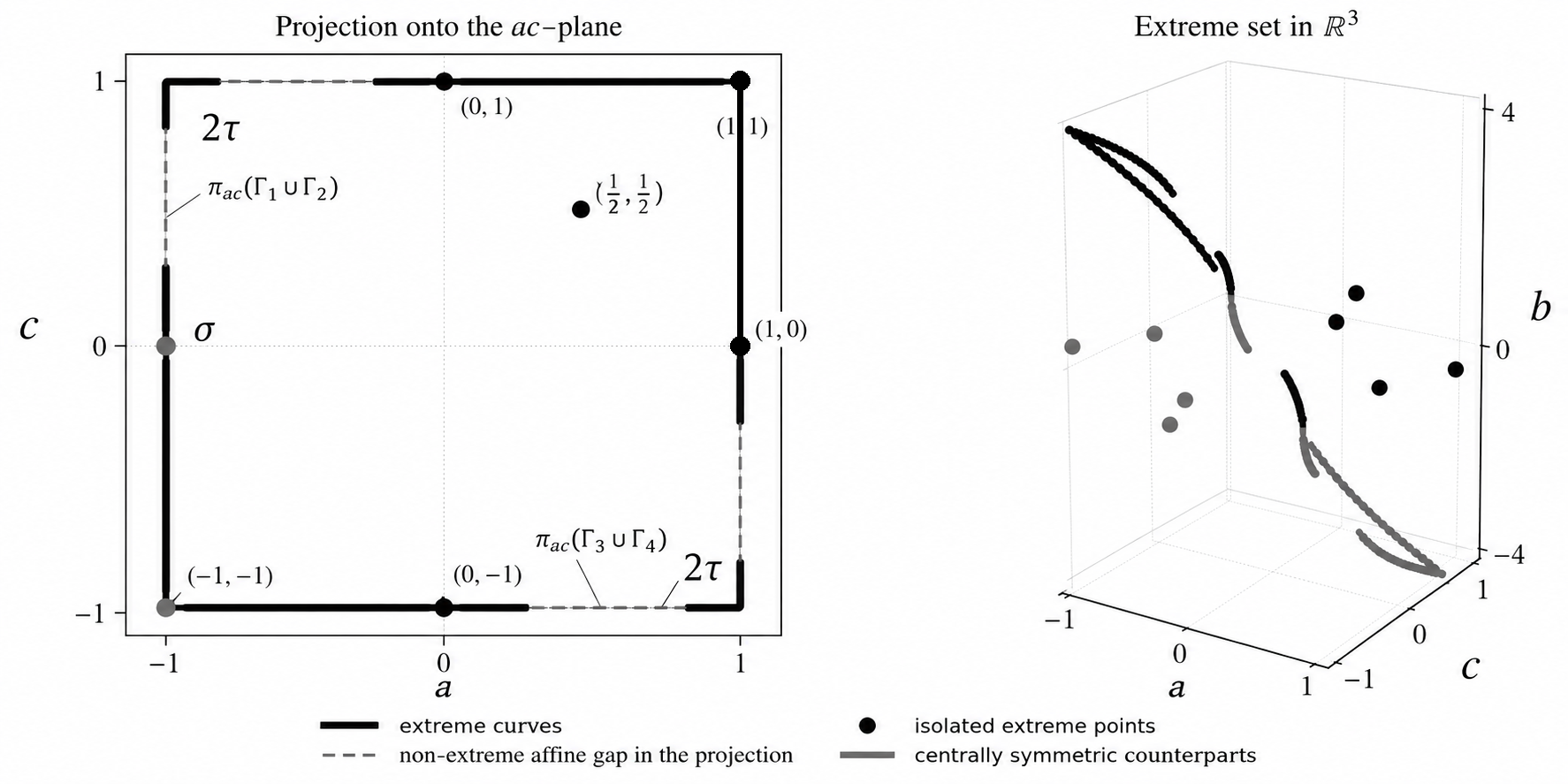}
\caption{The complete extreme set and its projection.  The left panel shows \(\pi_{ac}(\ext(\BH))\).  The solid black portions of the boundary are the projections of the curves \(\Gamma_1,\ldots,\Gamma_4\) and their negatives; the four dashed portions are the projected affine gaps and contain no relative-interior extreme points.  The black dots are the projections of \(p_0,\ldots,p_3\), and the gray dots are the projections of their negatives.  The right panel displays the centrally symmetric set \(\ext(\BH)\) in \(\R^3\): the upper curves and isolated points are black, and their negatives are gray.}
\label{fig:extreme-set}
\end{figure}

\begin{remark}
The families \(\Gamma_1,\ldots,\Gamma_4\) are genuine one-parameter curves of extreme points.  They must not be confused with the affine generators of the ruled surface pieces in \eqref{eq:A-generators}--\eqref{eq:C-generators}: the relative interiors of those generators are non-extreme.  This is why the unit-sphere plot in \cref{fig:sphere} is kept separate from the plot of the extreme set.
\end{remark}

\begin{remark}
The description in \cref{thm:extreme} reduces the maximization of any continuous convex functional on \(\BH\) to four one-parameter families and four pairs of isolated points.  This is the form needed for subsequent applications of the Krein--Milman method to sharp inequalities on \(\HH\).
\end{remark}

\FloatBarrier
\begingroup
\setlength{\emergencystretch}{2em}

\endgroup

\end{document}